\documentclass{amsart}

\usepackage{graphicx}
\usepackage{mathrsfs}
\usepackage{enumerate}
\usepackage{pgf,tikz}
\usepackage{xypic}
\usepackage{mathrsfs}
\usepackage{amsmath}
\usepackage{latexsym}
\usetikzlibrary{arrows}
\usepackage{wrapfig}
\usepackage{float}

\newcommand{\At}{\widetilde{\mathbb{A}}}

\newcommand{\N}{\mathbb{N}}
\newcommand{\Q}{\mathbb{Q}}
\newcommand{\Z}{\mathbb{Z}}

\newcommand{\g}{\gamma}

\newcommand{\eg}{\text{\textbf{EG}}}

\newcommand{\Aa}[2]{\mathscr{A}({#1},{#2})}

\newcommand{\X}{\mathscr{X}}

\newtheorem{theorem}{Theorem}[section]
\newtheorem{lemma}[theorem]{Lemma}

\newtheorem{crl}[theorem]{Corollary}

\theoremstyle{definition}
\newtheorem{definition}[theorem]{Definition}
\newtheorem{example}[theorem]{Example}
\newtheorem{conjecture}[theorem]{Conjecture}

\theoremstyle{remark}

\numberwithin{equation}{section}

\newcommand{\ti}[1]{\tilde{#1}}

\begin{document}

\title{Unistructurality of cluster algebras of type $\At$}

\author{V\'eronique Bazier-Matte}
\address{D\'epartement de mat\'ematiques, Universit\'e de Sherbrooke, 2500, boul. de l'Universit\'e, Sherbrooke, Qc J1K~2R1, Canada}
\email{Veronique.Bazier-Matte@usherbrooke.ca}

\begin{abstract}
It is conjectured by Assem, Schiffler and Shramchenko in \cite{ASS14} that every cluster algebra is unistructural, that is to say, that the set of cluster variables determines uniquely the cluster algebra structure. In other words, there exists a unique decomposition of the set of cluster variables into clusters. This conjecture has been proven to hold true for algebras of type Dynkin or rank 2 \cite{ASS14}. The aim of this paper is to prove it for algebras of type $\At$. We use triangulations of annuli and algebraic independence of clusters to prove unistructurality for algebras arising from annuli, which are of type $\At$. We also prove the automorphism conjecture \cite{ASS14} for algebras of type $\At$ as a direct consequence.
\end{abstract}

\thanks{The author would like to thank Pr. Ibrahim Assem for his support and guidance. The author also wishes to thank Prs Vasilisa Shramchenko and Ralf Schiffler and M. Guillaume Douville for fruitful discussions. Finally, the author would like to express her gratitude to the NSERC for the Canada Graduate Scholarships in Master’s Program.}


\maketitle


\section{Introduction}

Cluster algebras were introduced in 2002 by Fomin and Zelevinsky in \cite{FZ02} in order to create an algebraic framework for  canonical bases and total positivity.  Since then, cluster algebras have been proved to relate to various areas of mathematics, for example, combinatorics, representation theory of algebras, Lie theory, Poisson geometry and Teichm\"uller theory. 

Cluster algebras are commutative $\Z$-algebras with a distinguished set of generators, called cluster variables, grouped into sets called clusters. The set of all cluster variables is constructed recursively from a set of initial cluster variables using an operation called mutation. Every mutation defines a new cluster differing from the preceding one by a single cluster variable. A conjecture from Assem, Schiffler and Shramchenko in \cite{ASS14} states that if two clusters generate the same cluster algebra, then each cluster is obtained from the other by a sequence of mutations, that is, there exists a unique decomposition of the set of cluster variables into clusters. It is then said that such a cluster algebra is unistrucutural. The authors already proved it for cluster algebras of rank 2 or of Dynkin type. In this paper, we prove that cluster algebras of type $\At$ are unistructural.

To do so, we use triangulations of marked surfaces, defined by \cite{FST08} and lifted triangulations of marked surfaces \cite{ST09, Sch10}. Indeed, cluster variables are in one-to-one correspondence with isotopy classes of certain curves in the surface that are called arcs. Clusters are in bijection with triangulations, which are maximal sets of compatible arcs. Hence, we study cluster algebras through triangulations of annuli.

Finally, we prove the automorphism conjecture, formulated in \cite{ASS14}, for cluster algebras of type $\At$ as a direct corollary of the unistructurality of these algebras. This conjecture states that, for $\mathscr{A}$ a cluster algebra, $f:\mathscr{A} \rightarrow \mathscr{A}$ is a cluster automorphism if and only if $f$ is an automorphism of the ambient field which restricts to a permutation of the set of cluster variables. This conjecture is related to Conjecture 7.4(2) in \cite{FZ07} of Fomin and Zelevinsky, restated by \cite{ASS14} as follows. Two cluster variables $x$ and $x'$ in a cluster algebra $\Aa{X}{Q}$ are compatible if and only if for each cluster $Y$ of $\Aa{X}{Q}$ containing $x$, it is possible to write $x'$ as a Laurent polynomial in $Y$ of the reduced form $\frac{P}{M}$ where $P$ is a polynomial in the variables of $Y$ and $M$ is a monomial in the variables of $Y$ excluding $x$. This conjecture has been proven for cluster algebras arising from surfaces, so, in particular, for cluster algebras of type $\At$, see \cite{FST08}.

\section{Preliminaries}

\subsection{Cluster algebras}
A \emph{quiver} $Q$ is a quadruple $(Q_0, Q_1, s, t)$ such that $Q_0$ is a set whose elements are called \emph{points}, $Q_1$ is a set whose elements are called \emph{arrows} and $s$, $t:~Q_1 \rightarrow Q_0$ are two maps which associate to each arrow two points respectively called its \emph{source} and its \emph{target}. For $i \in Q_0$, we denote $i^{+}$ the set of arrows in $Q_1$ of source $i$. Similarly, $i^{-}$ is the set of arrows in $Q_1$ of target $i$
As part of the study of cluster algebras, one looks only at quivers with neither loops (arrows $\alpha \in Q_1$ such that $s(\alpha) = t(\alpha)$) nor 2-cycles (pairs of arrows $(\alpha, \beta) \in Q_1 \times Q_1$ such that $t(\alpha) = s(\beta)$ and $t(\beta) = s(\alpha)$).

Let $K$ be a field and $k$ a subfield of $K$. A set $Y = \left\{ y_1, \dots, y_n \right\} \subseteq K$ is \emph{algebraically independent} over $k$ if, for all polynomials $f \in k\left[X\right]$, $f \neq 0$, we have $f(y_1, \dots, y_n) \neq 0$.
A \emph{seed} $(X,Q)$ consists of a quiver $Q$ whose points are $Q_0 = \{1,2,\dots n \}$ and a set $X=\{x_1, \dots x_n\}$, algebraically independent over $\mathbb{Q}$, where we agree that $x_i$ corresponds to the point $i$. The set $X$ is called a \emph{cluster}.
The \emph{mutation of a seed} $(X,Q)$ in direction $x_k$ (or in direction $k$ if there is no ambiguity) transforms $(X,Q)$ into a new seed $\mu_{x_k}(X,Q) = (\mu_{x_k}(X), \mu_k(Q))$ where $\mu_k(Q)$ is a quiver obtained by achieving the followings steps on $Q$:
\begin{enumerate}
\item reverse all arrows incident to $k$;
\item for all pairs of arrows $\alpha, \beta \in Q_1$ such that $t(\alpha) = k = s(\beta)$, add a new arrow $\delta$ such that $s(\delta) = \alpha$ and $t(\delta) = \beta$;
\item remove one-by-one the ensuing 2-cycles,
\end{enumerate}
and $\mu_{k_x}(X) = X \setminus \{x_k\} \cup \{x_k'\}$ with $x_k'$ such that
\[
x_kx_k' = \prod_{\alpha \in k^{+}} x_{t(\alpha)} + \prod_{\alpha \in k^{-}} x_{s(\alpha)}.
\]
An empty product is put equal to 1. We call the equality above the \emph{exchange relation}.
We denote by $\mathscr{X}$ the union of all possible clusters obtained from the seed $(X,Q)$ by successive mutations and we call its elements \emph{cluster variables}. Two cluster variables are said to be \emph{compatible} if there exists a cluster containing both. The \emph{cluster algebra} $\mathscr{A}(X,Q)$ is the $\Z$-subalgebra of the \emph{ambient field} $\mathcal{F} = \Q(x_1,\dots,x_n)$ generated by $\mathscr{X}$.

Observe that quiver mutation is involutive and, hence, mutations of quivers induce an equivalence. It is said that two quivers $Q$ and $R$ are \emph{mutation-equivalent} if there exists a sequence of mutations $\mu = \mu_{k_n} \dots \mu_{k_1}$ such that $R = \mu(Q)$. A cluster algebra $\Aa{X}{Q}$ is of \emph{type $\At_{p,q}$} if $Q$ is mutation-equivalent to the following quiver:
\begin{figure}[H]
\includegraphics{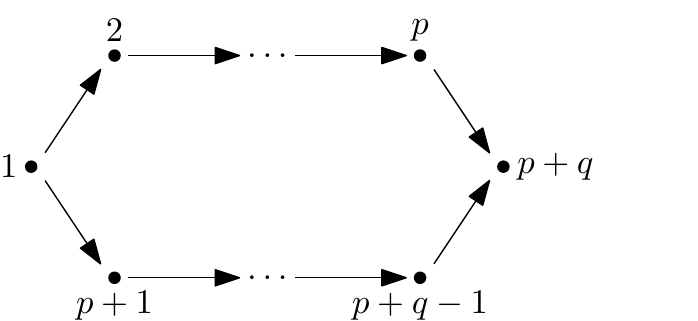}
\end{figure}

In order to prove the unistructurality of cluster algebras of type $\At$, we need a few properties of cluster algebras. One of them is the Laurent phenomenon, see \cite{FZ02}. It asserts that, for any cluster algebra $\mathscr{A}(X,Q)$ and for any cluster $Y = \{y_1, \dots, y_n \}$ in this algebra, each cluster variable can be written as a Laurent polynomial in $Y$, that is, for all $x \in \X$,
\[x = \frac{P(y_1, \dots, y_n)}{y_1^{d_1} \dots y_n^{d_n}}\] 
with $P \in \Z[y_1,\dots, y_n]$ and $d_i \in \N$. In particular,
\[ \X \subseteq \Z[y_1,\dots, y_n, y_1^{-1}, \dots, y_n^{-1}]. \]

Another interesting result about cluster algebras is the positivity theorem \cite{LS15}. It states that the numerator $P$ of a cluster variable written as a reduced Laurent polynomial has only non-negative coefficients. In particular, $P \in \N[y_1,\dots, y_n]$ and
\[ \X \subseteq \N[y_1,\dots, y_n, y_1^{-1}, \dots, y_n^{-1}]. \]

\subsection{Unistructurality}

Recall the definition of an \emph{exchange graph} $\eg(X,Q)$ introduced \cite{FZ02} for a seed $(X,Q)$. The set of vertices $\eg_0(X,Q)$ is the set of clusters in $\mathscr{A}(X,Q)$ and two clusters are joined by an edge if and only if one can be obtained from the other by a mutation, that is, they differ by a single cluster variable.

\begin{definition} \cite{ASS14}
We say that the cluster algebra $\mathscr{A}(X,Q)$ is \emph{unistructural} if, for any subset $Y$ of cluster variables of $\mathscr{A}(X,Q)$ and for any quiver $R$ such that $(Y,R)$ generates the same set of cluster variables, then $\eg(X,Q) = \eg(Y,R)$ and the sets of clusters of $\Aa{X}{Q}$ and of $\Aa{Y}{R}$ are the same. More precisely, if $(Y,R)$ generates by all possible successive mutations a family of cluster variables $\mathscr{Y}$, then
\begin{enumerate}
\item the equality $\mathscr{Y} = \mathscr{X}$ implies that $\eg(Y,R) = \eg(X,Q)$
\item there exists a permutation $\sigma$ of $\eg_0(X,Q)$ such that $X_{\alpha} = Y_{\sigma(\beta)}$ for any $\alpha \in \eg_0(X,Q)$ and with $\beta \in \eg_0(Y,R)$.
\end{enumerate}
\end{definition}

We already know that cluster algebras of rank 2 and of Dynkin type are unistructural \cite{ASS14}. The next conjecture is from \cite{ASS14}.

\begin{conjecture} \label{unisructurality}
All cluster algebras are unistructural.
\end{conjecture}

\subsection{Triangulations and universal covers}
Rather than studying cluster algebras of type $\At$ by performing mutations on their quivers, we study triangulations of annuli by flipping their arcs. The following definitions are adapted from \cite{FST08}.

A \emph{marked surface} $(S,M)$ consists of  a connected oriented 2-dimensional Riemann surface $S$ with boundary and a nonempty  finite set of \emph{marked points} $M$ in the closure of $S$ with at least one marked point on each boundary
component. Marked points in the interior of $S$ are called \emph{punctures}. An \emph{arc} $\gamma$ in $(S,M)$ is the isotopy class of a curve such that:
\begin{itemize}
\item the endpoints of the curve are in $M$;
\item the curve does not selfintersect, except that its endpoints may coincide;
\item the curve is disjoint from $M$ and the boundary of $S$ except for its endpoints;
\item the curve does not cut out an unpunctured monogon, that is the curve is not contractible into a point
of $M$.
\end{itemize}

An arc is \emph{interior} if it is not (isotopic to) an arc lying on the boundary of $S$ and it is a \emph{boundary arc} otherwise. Two arcs are \emph{compatible} if there are curves in their respective isotopy classes which do not intersect, except perhaps at their endpoints. A \emph{triangulation} is a maximal collection of pairwise compatible arcs.

Let $T$ be a triangulation of an unpunctured marked surface. A \emph{flip} $\mu_{\g}$ in $T$ replaces one particular arc $\gamma$  by the unique arc $\gamma'\neq\gamma$ which, together with the remaining arcs of $T$, forms a new triangulation. Because flips are clearly involutive, we call two triangulations \emph{flip-equivalent} if they are connected by a sequence of flips.

One can associates a quiver $Q_T$ to a triangulation $T$ of an unpunctured marked surface $(S,M)$ as follows. 
The points of $Q_T$ correspond to the interior arcs of $T$.
Let $\gamma_1$ and $\gamma_2$ be two sides of a triangle in $T$. There is an arrow $\alpha$ such that $s(\alpha) = \gamma_1$ and $t(\alpha) = \gamma_2$ if and only if $\gamma_2$ follows $\gamma_1$ with respect to the orientation of $S$ around their common vertex.
The cluster algebra $\Aa{X}{Q_T}$ given by a cluster $X$ and the quiver $Q_T$ is called the \emph{cluster algebra associated to the surface} $(S,M)$.
If $T_k$ is obtained from $T$ by a flip replacing an arc labeled $k$, then, $Q_{T_k} = \mu_k(Q_T)$ (see \cite{FST08}).
Indeed, one can show that $\Aa{X}{Q_T}$ does not depend on the particular triangulation $T$ we started from \cite{FST08}.
It is therefore possible to study cluster algebras associated to a marked surface through triangulations of this surface. More precisely, clusters are in one-to-one correspondence with triangulations and cluster variables are in one-to-one correspondence with arcs.

The exchange relation arising from flipping $\gamma$ in this figure
\begin{figure}[h!]
\includegraphics{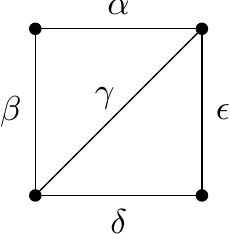}
\end{figure}
is the \emph{Ptolemy relation} $x_\gamma x_{\gamma'}=x_\alpha x_\delta + x_\beta x_\epsilon$, where $x_\eta$ denotes the cluster variable associated to the arc $\eta$ (if some of the sides are boundary arcs, the corresponding variables are replaced by 1).
We call such a set of arcs $\alpha$, $\beta$, $\delta$ and $\epsilon$ a \emph{quadrilateral} and the arcs $\g$ and $\g'$ its \emph{diagonals}.

If $Q_T$ is a quiver of type $\At_{p,q}$, then $T$ is the triangulation of an annulus with $p$ points on one boundary and $q$ points on the other, see \cite{ABCP10}. In an annulus, arcs with both endpoints on the same boundary are called \emph{peripherical} while arcs with endpoints on different boundaries are called \emph{bridging}.

\begin{example}
\begin{figure}[h!]
\includegraphics{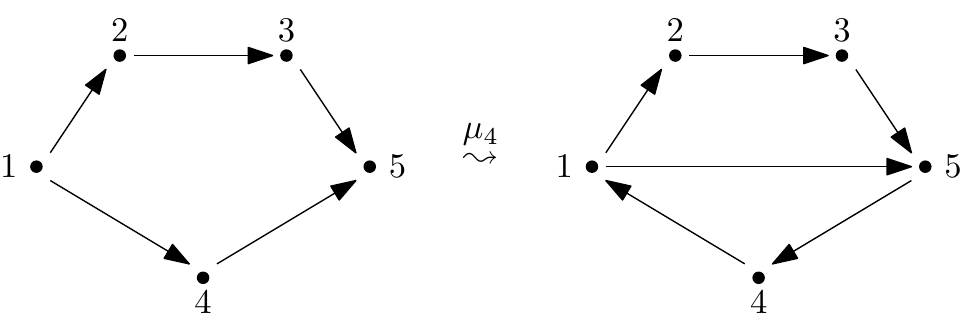}
\caption{Let $Q$ be a quiver of type $\At_{3,2}$, the first quiver shown. We mutate it in direction 4 to obtain the second quiver.}
\label{A_3,2}
\end{figure}
\begin{figure}[h!]
\includegraphics{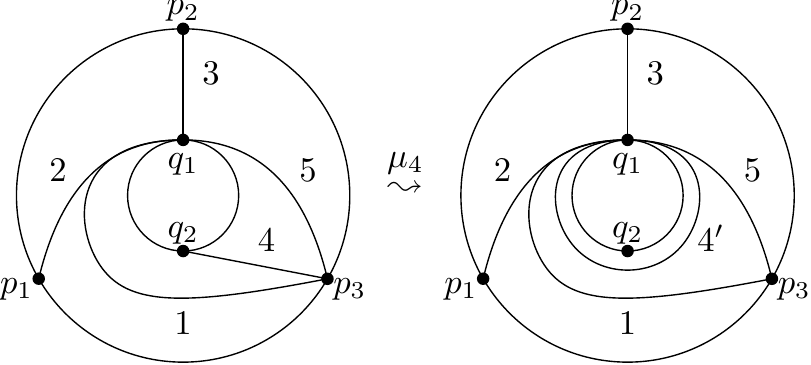}
\caption{Let $T$ be a triangulation of an annulus with 3 points on a boundary and 2 points on the other, the first triangulation shown. We obtain the second triangulation, $T'$, by flipping arc 4.}
\label{T-A_3,2}
\end{figure}
The quiver $Q$ in the figure \ref{A_3,2} and the triangulation $T$ in the figure \ref{T-A_3,2} are associated, that is $Q = Q_T$. Moreover, $\mu_4(Q)$ and $T'$ are associated.

Let associate the cluster variable $x_i$ to arc $i$. By the Ptolemy relation, we find $x_4 x_4' = x_1 \cdot 1 + 1 \cdot x_5$.

The arc 4 in $T$ is bridging, while its flip, the arc $4'$ in $T'$, is peripherical. Note that $4'$ is not isotopic to the inside boundary of the annulus, because of the another marked point on it.
\end{example}

Representing and visualizing arcs in annuli can rapidly be confusing if they wrap around the inside boundary. Therefore, one prefers to use a universal cover. 
Let $Q$, $\tilde{Q}$ be two quivers such that $\Z$ acts freely on $\tilde{Q}$, that is $n \cdot \tilde{i} \neq \ti{i}$ for all points $\ti{i} \in \ti{Q}_0$ and all $n \in \Z^{*}$.
Note $Q(i,j) = \left\lbrace \alpha \in Q \mid s(\alpha) = i \text{ and } t(\alpha) = j \right\rbrace$ and $\tilde{Q}(\tilde{i},\tilde{j}) = \left\lbrace \tilde{\alpha} \in \tilde{Q} \mid s(\tilde{\alpha}) = \tilde{i} \text{ and } t(\tilde{\alpha}) = \tilde{j} \right\rbrace$ for all $i$, $j \in Q_0$ and for all $\tilde{i}$, $\tilde{j} \in \tilde{Q}_0$.
A \emph{Galois covering} $\pi : \tilde{Q} \rightarrow Q$ is given by two maps	 $\pi_0 : \tilde{Q}_0 \rightarrow Q_0$ and $\pi_1 : \tilde{Q}_1 \rightarrow Q_1$ such that: \begin{itemize}
\item for all points $\tilde{i}$, $\tilde{j} \in \tilde{Q}_0$, we have $\pi_1 \left( \tilde{Q}_1(\tilde{i}, \tilde{j}) \right) \subseteq Q_1 \left( \pi_0 \left(\tilde{i} \right), \pi_0 \left(\tilde{j}\right) \right)$;
\item the map $\pi_0$ is surjective;
\item for all $n \in \Z$, we have $\pi \circ n = \pi$;
\item if $\tilde{i}$, $\tilde{j} \in \tilde{Q}_0$ such that $\pi_0(\tilde{i}) = \pi_0(\tilde{j})$, then there exists $n \in \Z$ such that $\tilde{j} = n \cdot \tilde{i}$;
\item for all $\ti{i} \in \ti{Q}_0$, $\pi_1$ induces two bijections $\ti{i}^{+} \rightarrow \left( \pi_0\left(\ti{i}\right) \right)^{+}$ and $\ti{i}^{-} \rightarrow \left( \pi_0\left(\ti{i}\right) \right)^{-}$.
\end{itemize}
If $Q = Q_T$ and $\tilde{Q} = \tilde{Q}_{\tilde{T}}$, it induces a map $\pi' : \tilde{T} \rightarrow T$. The triangulation $\tilde{T}$ is called the \emph{lifted triangulation} of $T$.

The elements of the fiber $\pi^{-1}(i)$ of a point $i$ in $Q_0$ are of the form $\ti{i}_n$ with $n \in \Z$. Let $\tilde{\mu}_i = \prod_{n \in \Z} {\mu_{\ti{i}_n}}$. Then, $\tilde{\mu}_i (\tilde{T}) = (\pi')^{-1}(\mu_i(T))$, see \cite{Sch10}.

\begin{example}
Triangulations in Figure \ref{T-A_3,2-revetement} are the lifted triangulations of triangulations in Figure \ref{T-A_3,2}.
\begin{figure}[h]
\includegraphics[]{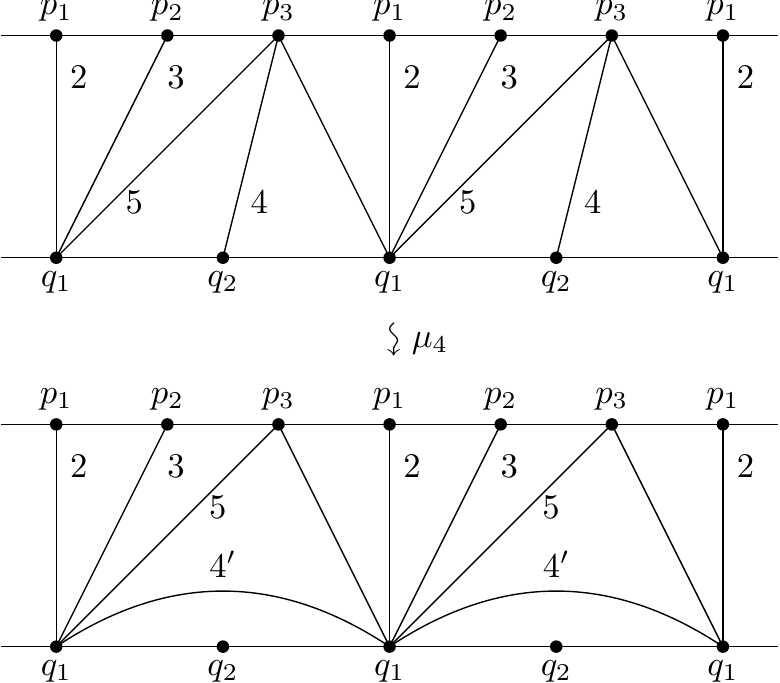}
\caption{Lifted triangulations associated to triangulations in Figure \ref{T-A_3,2}}
\label{T-A_3,2-revetement}
\end{figure}
\end{example}


\section{Unistructurality for $\At$}

In order to prove the theorem \ref{yeah}, we need this technical lemma.

\begin{lemma}\label{Laurent}
Let $x_1$ and $x_2$ be two variables from a cluster algebra $\mathscr{A}(X,Q)$ and let $\Sigma_1$, $\Sigma_2$ and $\Sigma_3$  be sums of products of cluster variables with positive coefficients such that $x_1x_2$ is not a term of these sums or is not the only one.
\begin{enumerate}[a)]
\item If $x_1x_2 = \Sigma_1$, then $x_1 \not\in X$ or $x_2 \not\in X$.
\item If $x_1x_2\Sigma_1 = \Sigma_1\Sigma_2 + \Sigma_3$, then $x_1$ or $x_2 \not\in X$.
\end{enumerate}
\end{lemma}

\begin{proof} \hfill
\begin{enumerate}[a)]
\item Each cluster variable of $\mathscr{A}(X,Q)$ can be written as a Laurent polynomial in $X$. Therefore, $\Sigma_1$, as a sum of products of cluster variables, can also be written as a Laurent polynomial, say $\frac{P_1}{M_1}$, where $P_1$ is a polynomial in $X$ and $M_1$ is a monomial in $X$. We now have  $x_1x_2 = \frac{P_1}{M_1}$, which is equivalent to \[M_1x_1x_2 -P_1 = 0.\]
If $x_1 \in X$ and $x_2 \in X$, we would get a contradiction to the algebraic independence of $X$.

\item Let $\Sigma_i = \frac{P_i}{M_i}$, where $P_i$ is a polynomial in $X$ and $M_i$ a monomial in $X$ for $i=1,2,3$. Thereby, we obtain $x_1x_2\frac{P_1}{M_1} = \frac{P_1}{M_1}\frac{P_2}{M_2} + \frac{P_3}{M_3}$ which gives \[M_1M_2P_3 = x_1x_2M_2M_3P_1 -M_3P_1P_2 = M_3P_1 (x_1x_2M_2- P_2).\] By the positivity theorem, coefficients of the terms in $X$ in cluster variables are positive; it is in particular the case in $P_3$. Since there is a minus sign on the right side of the equality, we deduce that, either $M_3P_1P_2 = 0$, which is impossible, or each of the terms of the development of  $M_3P_2$ appears in the development of $x_1x_2M_2$ with a coefficient greater or equal.

Suppose that $x_1 \in X$ and $x_2 \in X$; in particular, they are monomials. Since $x_1x_2M_2$ is only a product of variables of $X$ without any coefficient, we deduce that $P_2 = x_1x_2M_2$, thus, $M_1M_2P_3 = 0$, which is once again a contradiction to the algebraic independence of $X$. \qedhere
\end{enumerate}
\end{proof}

\begin{theorem}\label{yeah}
Cluster algebras of type $\At$ are unistructural.
\end{theorem}

\begin{proof}
Let $X=\{x_1,\dots,x_{p+q}\}$ and $Y$ be algebraically independent sets on $\Q$, $Q$ a quiver of type $\At_{p,q}$ and $R$ a quiver such that the cluster variables of $\mathscr{A}(X,Q)$ and of $\mathscr{A}(Y,R)$ are the same. Remark that a consequence of that hypothesis is that $X$ et $Y$ have the same number of elements, which is the cardinality of a transcendence basis of the (common) ambient field. We show that $Y$ is a cluster of $\mathscr{A}(X,Q)$.

Assume to the contrary that $Y$ is not a cluster of $X$. Thus, there exists two variables $y_i$ and $y_j$ of $Y$ incompatible in $\mathscr{A}(X,Q)$. We know that variables of $\mathscr{A}(X,Q)$ are associated with arcs of an annulus with $p$ marked points on a border and $q$ points on the other one.
Consider the arcs $\g_i$ and $\g_j$ associated to $y_i$ and $y_j$ in $\mathscr{A}(X,Q)$. Let $i_1$ and $i_2$ be the endpoints of $\g_i$ and $j_1$ and $j_2$ be the endpoints of $\g_j$.
We have to study three cases: the arcs $\g_i$ and $\g_j$ can both be periphical, both bridging or one is periphical and the other is bridging.

\textbf{ First case: the arc $\g_i$ is periphical and the arc $\g_j$ is bridging.} Without loss of generality, assume that $i_1$, $i_2$ and $j_1$ are all on the same border. Consider the triangulation of Figure \ref{p1-t1} where $z_1$,  $z_2$, $z_3$ and $z_4$ are cluster variables of a same cluster of $\mathscr{A}(X,Q)$. Note that $z_1$ and $z_2$ could be associated to boundary arcs, in which case they are equal to $1$. Note also that it is possible that $i_1=i_2$.
\begin{figure}[htb]
\includegraphics{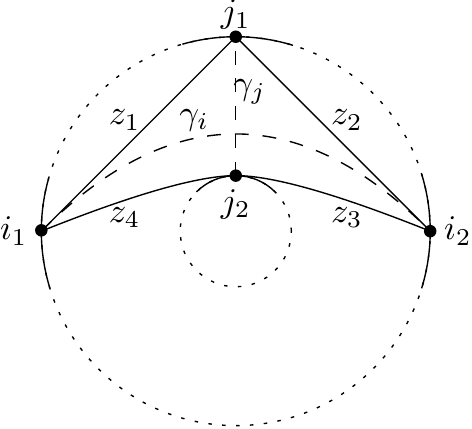}
\caption{A quadrilateral with an incompatible periphical arc and a bridging arc as diagonals}
\label{p1-t1}
\end{figure}

By the Ptolemy relation, we obtain
\[y_i y_j = z_1 z_3 + z_2 z_4. \]
Now, $z_1 z_3 + z_2 z_4$ is a sum of product of cluster variables in $\mathscr{A}(X,Q)$ so it is also the case in $\mathscr{A}(Y,R)$. Therefore,  Lemma \ref{Laurent} ensures that $y_i$ or $y_j \not\in Y$, which is a contradiction to the initial hypothesis.

\textbf{ Second case: the arcs $\g_i$ and $\g_j$ are peripherical.} If they intersect each other exactly once, there exists once again a quadrilateral with these arcs as diagonals, as shown in Figure \ref{qp2} (where possibly $i_1=j_2$), and we again obtain a contradiction to Lemma \ref{Laurent}.
\begin{figure}[htbp!]
\centering
\includegraphics{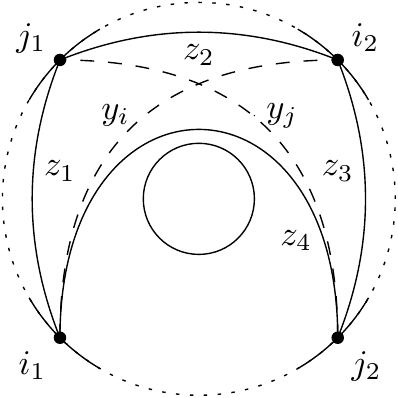}
\caption{A quadrilateral with peripherical arcs intersecting each other exactly once as diagonals}
\label{qp2}
\end{figure}

Now assume that $\g_i$ et $\g_j$ intersect each other twice. By representing them in the lifted triangulation of an annulus, one sees easily that they cannot intersect each other more than twice and that, in that case, $i_1$ and $i_2$ are necessarily distinct from $j_1$ and $j_2$. First, consider the case where $i_1$, $i_2$, $j_1$ and $j_2$ are all distinct. Let $Z$ be the cluster of $\mathscr{A}(X,Q)$ associated with the triangulation of Figure \ref{p2-u} (represented as a lifted triangulation) with $Z=\{z_1,\dots,z_{p+q}\}$ and $z_1 = y_i$.
\begin{figure}[h]
\includegraphics{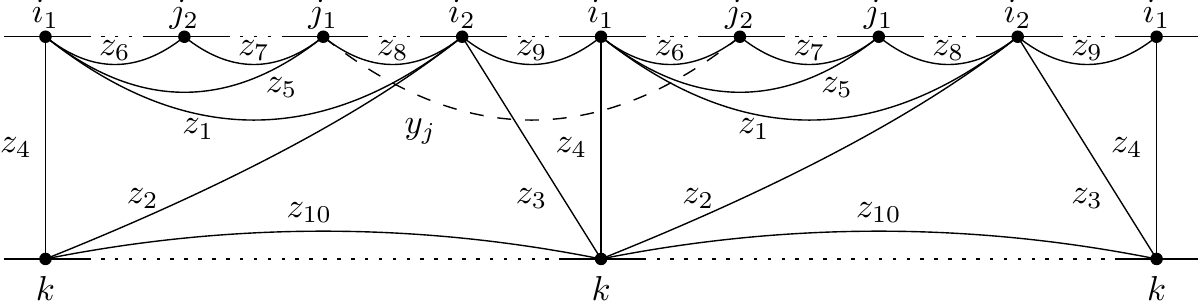}
\caption{A triangulation containing the peripherical arc $\gamma_i$ intersecting twice the peripherical arc $\gamma_j$}
\label{p2-u}
\end{figure}

Remark that $z_6, \dots, z_{10}$ could be associated with boundary arcs.
Consider the sequence of mutations in order to obtain $y_j$ as shown in figure \ref{p2}.

\begin{figure}[htbp!]
\centering
\includegraphics[trim=66pt 0pt 80pt 0pt, clip]{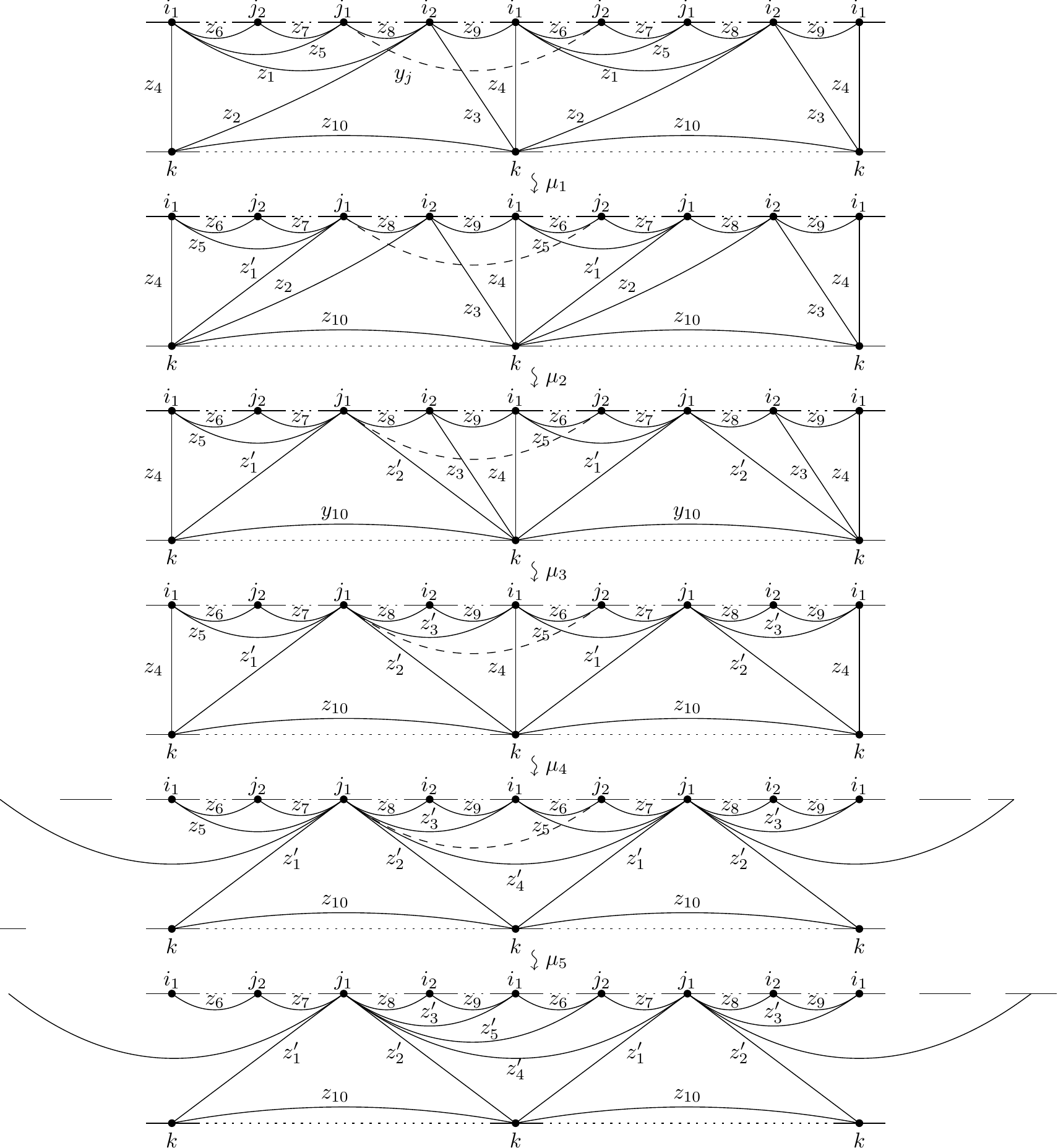}
\caption{A sequence of mutation from a triangulation containing the peripherical arc  $\gamma_i$ associated with $y_i=z_1$ to the peripherical arc $\g_j$ associated to $y_j$ and intersecting $\g_j$ twice}
\label{p2}
\end{figure}
Remark that the last triangulation of Figure \ref{p2} implies $z_5' = y_j$. Moreover, because of the Ptolemy relation, we obtain the following five equalities: 
\begin{itemize}
\item $z_1z_1' = z_2z_5 + z_4z_8$;
\item $z_2z_2' = z_1'z_3 + z_8z_{10}$;
\item $z_3z_3' = z_2'z_9 + z_4z_8$;
\item $z_4z_4' = z_1'z_3' + z_2'z_5$;
\item $z_5z_5' = z_3'z_7 + z_4'z_6$.
\end{itemize}
From these equalities, we develop the expression $z_1z_1'z_2'z_5'$ :
\begin{align*}
z_1z_1'z_2'z_5' & = (z_2z_5 + z_4z_8)z_2'z_5'\\
 & = (z_2z_2')(z_5z_5') + \Sigma_1\\
 & = (z_1'z_3 + z_8z_{10})(z_3'z_7 + z_4'z_6) + \Sigma_1\\
 & = (z_3z_3')z_1'z_7 + \Sigma_2\\
 & = (z_2'z_9 + z_4z_8)z_1'z_7 + \Sigma_2\\
 & = (z_1'z_2')z_7z_9 + \Sigma_3,
\end{align*}
where 
\begin{align*}
&\Sigma_1 = z_2'z_4z_5z_8,\\
&\Sigma_2 = z_1'z_3z_4z_6 + z_3'z_7z_8z_{10} + z_4'z_6z_8z_{10} + \Sigma_1 \text{ and} \\
&\Sigma_3 = z_1'z_4z_7z_8 + \Sigma_2.
\end{align*}
Recall that every $z_k$ and $z_k'$, being cluster variables of $\mathscr{A}(X,Q)$,  are also cluster variables of $\mathscr{A}(Y,R)$, so are Laurent polynomial in $Y$. The same applies to $\Sigma_3$ since it is a sum of products of cluster variables. By Lemma \ref{Laurent}, $y_i \not\in $ or $y_j \not\in Y$ (because $y_i=z_1$ and $y_j=z_5$), which is a contradiction.

If we have $i_1=i_2$ or $j_1=j_2$, we would arrive at the same contradiction, but the sequence of flips to achieve it would be shorter because some arcs would be equal.

\textbf{Third case: the arcs $\gamma_i$ and $\gamma_j$ are bridging.} As a first step, consider the case where  $i_1$ and $j_1$, as well as $i_2$ and $j_2$, are distinct, and where the arcs $\gamma_i$ are $\gamma_j$ intersect each other exactly once. Then, there exists a quadrilateral with $\g_i$ and $\g_j$ as diagonals shown in triangulation of the Figure \ref{t2} and we obtain the same contradiction as in the first case.

\begin{figure}[htbp!]
\includegraphics{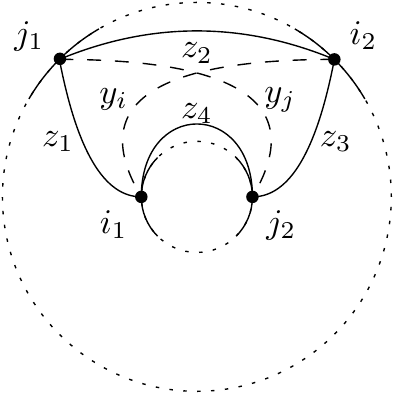}
\caption{A quadrilateral with two arcs bridging intersecting each other once as diagonals}
\label{t2}
\end{figure}

Now consider the cases where $\g_i$ et $\g_j$ intersect each other $n$ times with $n \geq 2$. Let $Z$ be the cluster of $\mathscr{A}(X,Q)$ associated to the triangulation in Figure \ref{n=2-1}

\begin{figure}[htbp!]
\centering
\includegraphics[trim=70pt 0pt 70pt 0pt, clip]{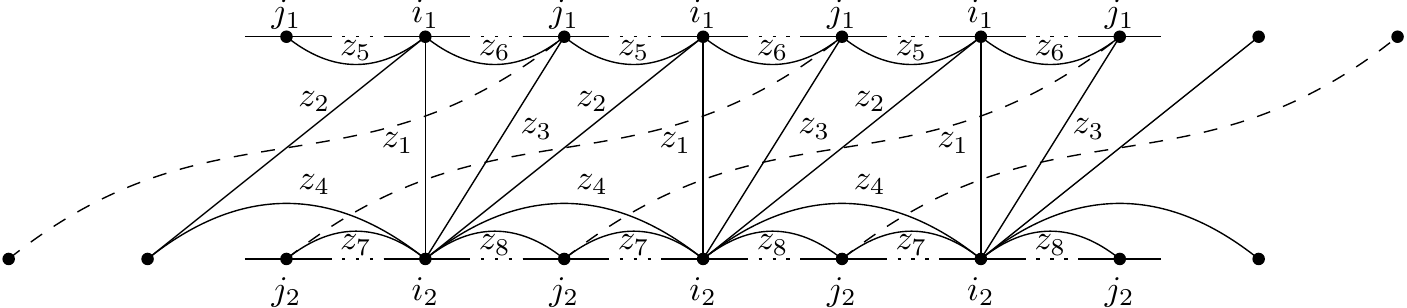}
\caption{A triangulation containing the bridging arc $\g_i$ associated to $y_i=z_1$ and intersecting $n=2$ times the bridging arc $\g_i$}
\label{n=2-1}
\end{figure}
There exists, up to symmetry, a unique arc $\g_j$ intersecting $\g_i$ $n$ times. For $n=2$, this arc is represented by a dotted line. After a sequence of flips, we obtain this arc (see Figure \ref{n=2}). 

\begin{figure}[htbp!]
\centering
\includegraphics[trim=76pt 0pt 152pt 0pt, clip]{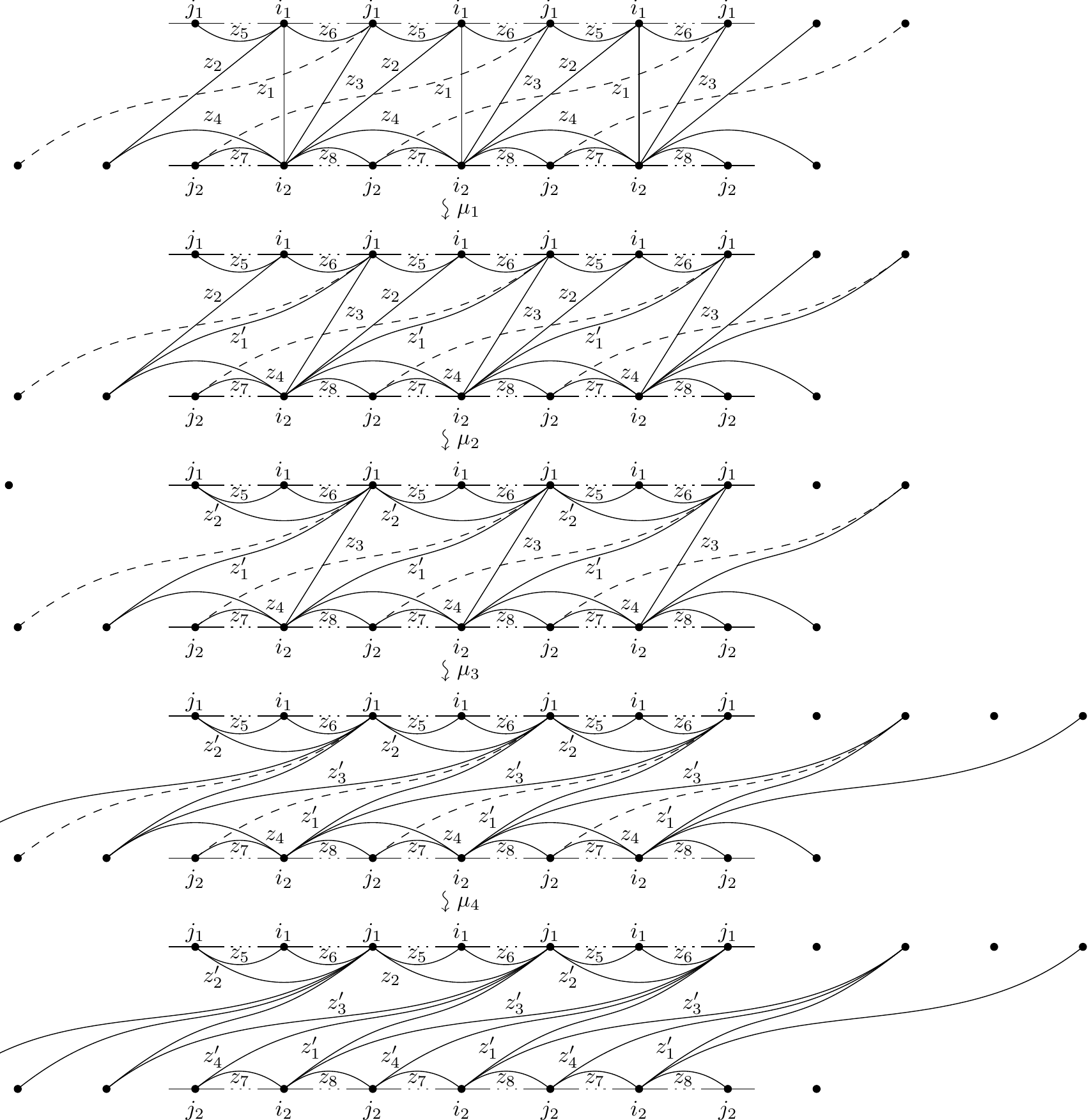}
\caption{A sequence of flips from a triangulation containing the bridging arc $\g_i$ associated to $y_i=z_1$ to a triangulation containing the bridging arc $\g_j$ associated to $y_j=z_4'$ and intersecting $n=2$ times $\g_i$}
\label{n=2}
\end{figure}

We obtain the arc $\g_j$ associated to $z_4' = y_j$. By the Ptolemy relation, we deduce the following equalities: 
\begin{itemize}
\item $z_1 z_1' = z_2 z_3 + z_4 z_6$;
\item $z_2 z_2' = z_1' z_5 + z_3 z_6$;
\item $z_3 z_3' = \left(z_1'\right)^2 + z_2' z_4$;
\item $z_4 z_4' = z_1' z_8 + z_3' z_7$.
\end{itemize}

We use some of these equalities to develop the expression $z_1 z_1' z_4'$.
\begin{align*}
z_1 z_1' z_4' &= (z_2 z_3 + z_4 z_6) z_4'\\
 & = (z_4 z_4') z_6 + \Sigma_1\\
 & = (z_1' z_8 + z_3' z_7) z_6 + \Sigma_1\\
 & = z_1' z_6 z_8 + \Sigma_2
\end{align*}
where
$\Sigma_1 = z_2 z_3 z_4'$ et $\Sigma_2 = z_3'z_6z_7 + \Sigma_1$.
Since $z_1 = x_i'$ et $z_4' = x_j'$, we obtain once again a contradiction as a consequence of Lemma \ref{Laurent}.

If the $\gamma_i$ and $\gamma_j$ intersect each other $n=3$ times, we have to achieve an additional mutation from the last triangulation of the Figure \ref{n=2} to obtain the arc $\g_j$; this mutation is shown in Figure \ref{n=3}. Note that, in this figure, the arc $\g_i$ is represented by a dotted line.

\begin{figure}[htbp!]
\centering
\includegraphics{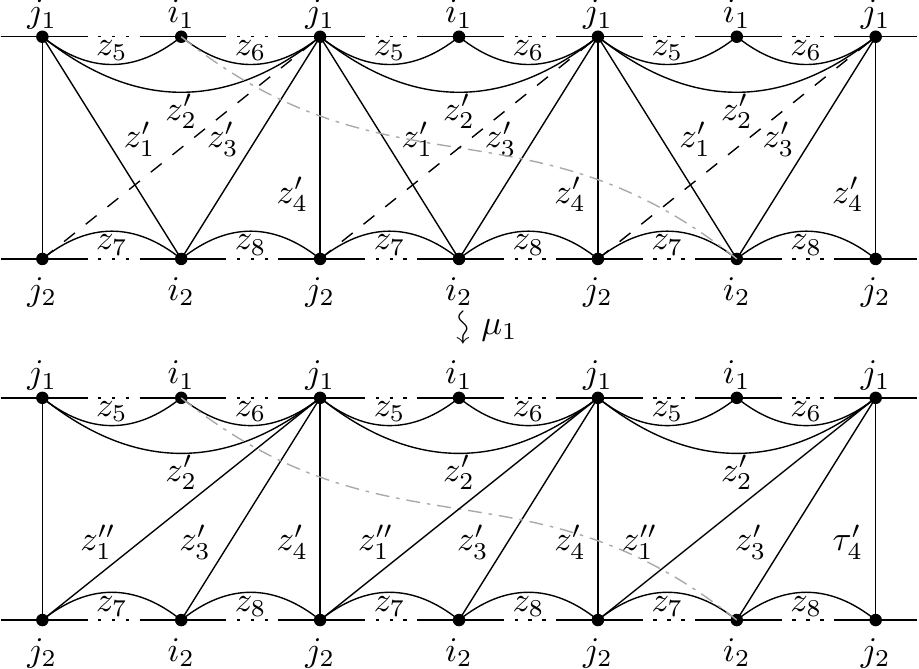}
\caption{Mutation to obtain peripherical arc $\g_j$ associated to $y_j=z_1''$ and intersecting $n=3$ times $\g_i$}
\label{n=3}
\end{figure}

From the mutations of variables, we obtain the equality $z_1' z_1'' = z_2'z_7 + z_3' z_4'$. We develop the product $z_1 z_1' z_3' z_1''$:
\begin{align}
z_1 z_1' z_3' z_1'' &= (z_2 z_3 + z_4 z_6) z_3' z_1'' \notag \\
& = (z_3 z_3') z_1'' z_2 + \Sigma_4 \notag\\
& = \left( (z_1')^2 + z_2'z_4 \right)z_1'' z_2 + \Sigma_4 \notag\\
& = (z_1' z_1'') z_1' z_2 + \Sigma_5 \notag\\
& = (z_2'z_7 + z_3' z_4') z_1' z_2 + \Sigma_5 \notag\\ 
& = (z_1' z_3') z_2 z_4' + \Sigma_6 \label{3.1}
\end{align}
where $\Sigma_4 = z_1''z_3'z_4z_6$, $\Sigma_5 = z_1'' z_2 z_2' z_4 + \Sigma_4$ and $\Sigma_6 = z_1' z_2 z_2' z_7 + \Sigma_5$.
Using Lemma \ref{Laurent}, we find once again $z_1 \not\in Y$ or $z_1'' \not\in Y$, so $y_i$ or $y_j \not\in Y$.

Now consider the case where $n=4$, that is where the arcs $\g_i$ and $\g_j$ intersect each other four times. We obtain $\g_j$ from last triangulation of Figure \ref{n=3} after two other mutations, shown at Figure \ref{n=4}.

\begin{figure}[htbp!]
\centering
\includegraphics[trim=70pt 0pt 70pt 0pt, clip]{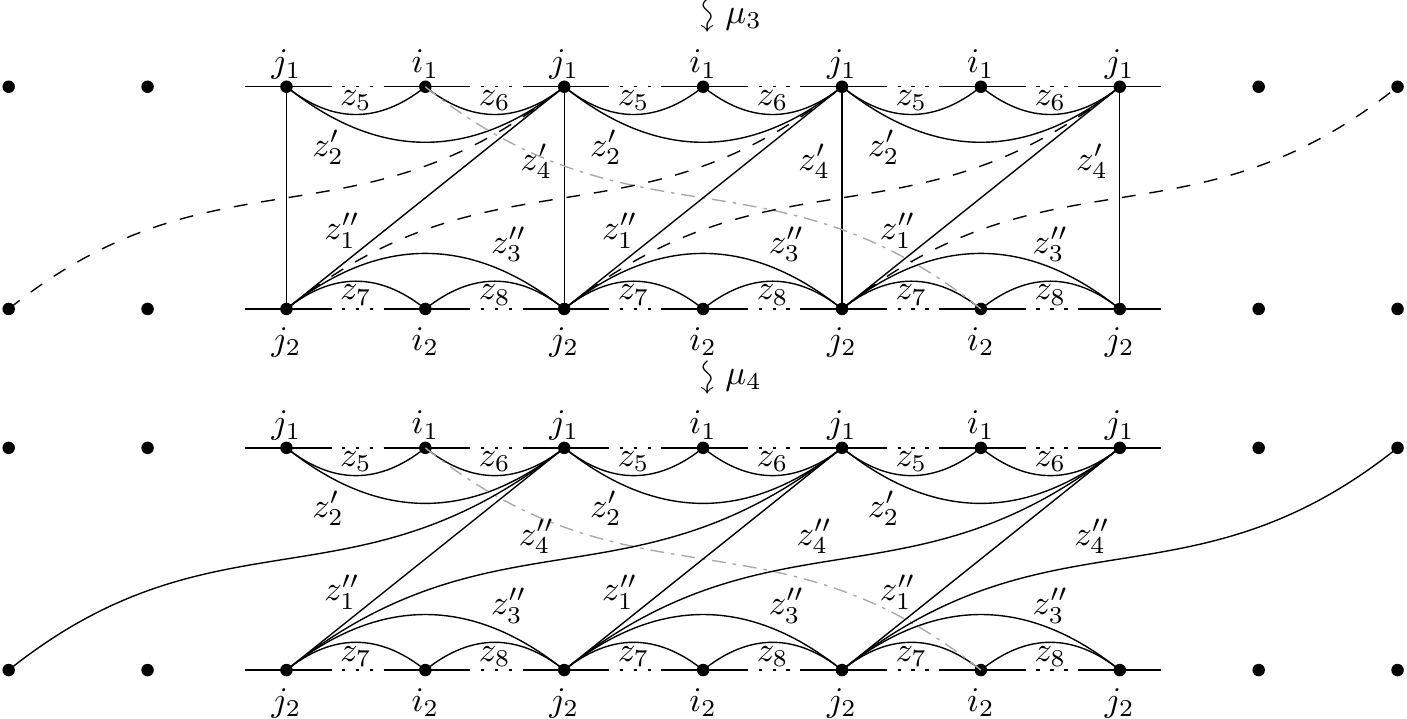}
\caption{A sequence of mutations to obtain the peripherical arc $\g_j$ associated with $y_j=z_4''$ and intersecting $n=3$ times $\g_i$}
\label{n=4}
\end{figure}

We obtain equalities 
$z_3' z_3'' = z_1''z_8 + z_4'z_7$
and $z_4'z_4'' =\left( z_1''\right)^2 + z_2'z_3''$
with the Ptolemy relation.

From now on, we can consider the general case where $\g_i$ and $\g_j$ intersect each other $n$ times, with $n\geq 4$.
Denote $\mu_0 = \mu_3 \mu_1 \mu_4 \mu_3 \mu_2 \mu_1$ (it is all mutations until now achieved except for the last one).
Denote also by $z^{(k)}_1$ et $z^{(k)}_4$ the last variables obtained by sequences of mutations $(\mu_1 \mu_4)^{k-2}\mu_0$ and $\mu_2(\mu_1 \mu_4)^{k-2}\mu_0$ respectively, with $k \geq 3$.

A simple induction yields 
\[z^{(k-1)}_1z^{(k)}_1 = \left(z^{(k-1)}_4\right)^2 + z_2'z_3''\]
and \[z^{(k-1)}_4z^{(k)}_4 = \left(z^{(k)}_1\right)^2 + z_2'z_3''.\]
Moreover, from Figure \ref{n=4}, remark by induction that the arc associated to $z_1^{(k)}$ intersects $2k-1$ times $\gamma_i$ and the arc associated to $z_2^{(k)}$ intersects it $2k$ times for $k \geq 3$.
Using equality \ref{3.1}, we obtain 
\begin{align*}
z_1 z_1' z_3' z_1'' z_4''
&= \left((z_1' z_3') z_2 z_4' + \Sigma_6\right) z_4'' \\
&= (z_4' z_4'') z_1' z_3' z_2 + z_4''\Sigma_6 \\
&= \left( \left( z_1''\right) ^2 + z_2'z_3'' \right) z_1' z_3' z_2 + z_4''\Sigma_6 \\
&= z_1' z_3' \left(z_1''\right)^2 z_2  + \Sigma_7,
\end{align*}
where $\Sigma_7 = z_1' z_2 z_2' z_3' z_3'' + z_4''\Sigma_6$.
Hence, we deduce by induction
\[z_1 z_1' z_3' \prod_{k=2}^{m-1} \left( z_1^{(k)}z_4^{(k)} \right) z_1^{(m)} = z_1'z_3'\prod_{k=2}^{m-1} \left( z_1^{(k)}z_4^{(k)} \right)\left(z_2^{(m-1)}\right)^2 + \Sigma_{2m+2} \]
and
\[z_1 z_1' z_3' \prod_{k=2}^m \left( z_1^{(k)}z_4^{(k)} \right) = z_1'z_3'\prod_{k=2}^{m-1} \left( z_1^{(k)}z_4^{(k)} \right)\left(z_1^{(m)}\right)^2 + \Sigma_{2m+3} \]
for $m \geq 3$ and where $\Sigma_{2m+2}$ and $\Sigma_{2m+3}$ are sums of products of cluster algebras with positive coefficients.

Using Lemma \ref{Laurent}, we conclude that $z_1 \not\in Y$ or $z_1^{(k)} \not\in Y$ for all $k \geq 3$. Similarly, $z_1 \not\in Y$ or $z_2^{(k)} \not\in Y$ for all $k \geq 3$. If $i_1=j_1$ or $i_2=j_2$, the proof technique would be similar, but we would consider less arcs in triangulations and, hence, there would be less mutations to achieve. Thus, if $y_j$ is associated with a bridging arc intersecting $\gamma_i$ (regardless of the number of crossings), $y_i \not\in Y$ or $y_j \not\in Y$.

Since we have proved that $y_i$ and $y_j$ cannot be associated with a peripherical arc and a bridging arc, to two peripherical arcs or to two bridging arcs intersecting each other in $\mathscr{A}(X,Q)$, then they are associated with two compatible arcs of $\mathscr{A}(X,Q)$. Hence, $y_i$ and $y_j$ are compatible variables in $\mathscr{A}(X,Q)$ and, since $|Y| = |X|$, $Y$ is a cluster of $\mathscr{A}(X,Q)$.

To show that cluster algebras of type $\At$ are unistructural, we still have to show that $R$ is a quiver of type $\At$. To do so, we adapt the proof from \cite{ASS14-2} for the Dynkin case. Let $Z = \{z_1, \dots z_{p+q}\}$ be a cluster of $\mathscr{A}(X,Q)$ associated to the following quiver, denoted by $Q_Z$:
\begin{figure}[H]
\includegraphics{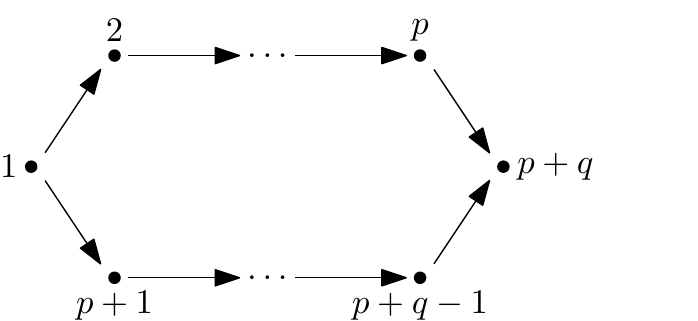}
\end{figure}
By the first part of this proof, we know that $Z$ is also a cluster $\mathscr{A}(Y,R)$, but not necessarily associated to the same quiver. Denote by $R_{Z}$ the quiver of $\mathscr{A}(Y,R)$ associated to $Z$ and $\mu_{Q_Z}(z_i)$ and $\mu_{R_Z}(z_i)$ the new variables created by mutation in direction $i$ of $Z$ in $\mathscr{A}(X,Q)$ and in $\mathscr{A}(Y,R)$ respectively. Both variables are Laurent polynomials in $Z$ with $z_i$ as denominator. More precisely,
\[ \mu_{Q_Z}(z_i) = \frac{ \prod\limits_{\alpha \in i^{+} \text{ in } Q_Z}z_{t(\alpha)} + \prod\limits_{\alpha \in i^{-}\text{ in } Q_Z}z_{s(\alpha)} } {z_i} \]
and
\[ \mu_{R_Z}(z_i) = \frac{ \prod\limits_{\alpha \in i^{+} \text{ in } R_Z}z_{t(\alpha)} + \prod\limits_{\alpha \in i^{-}\text{ in } R_Z}z_{s(\alpha)} } {z_i}. \]
As the quiver $Q_Z$ is acyclic, there exists a unique variable of $\mathscr{A}(X,Q)$ with $z_i$ as denominator, see theorem 2.2 in \cite{BMRT07}. This variable is also unique in $\mathscr{A}(Y,R)$, hence
\[ \prod\limits_{\alpha \in i^{+} \text{ in } Q_Z}z_{t(\alpha)} = \prod\limits_{\alpha \in i^{+} \text{ in } R_Z}z_{t(\alpha)} 
\text{ and }
\prod\limits_{\alpha \in i^{- }\text{ in } Q_Z}z_{s(\alpha)} = \prod\limits_{\alpha \in i^{-} \text{ in } R_Z}z_{s(\alpha)} 
\]
or 
\[ \prod\limits_{\alpha \in i^{+} \text{ in } Q_Z}z_{t(\alpha)} = \prod\limits_{\alpha \in i^{-} \text{ in } R_Z}z_{s(\alpha)}
\text{ and }
\prod\limits_{\alpha \in i^{-} \text{ in } Q_Z}z_{s(\alpha)} = \prod\limits_{\alpha \in i^{-} \text{ in } R_Z}z_{t(\alpha)}. \]
As $i$ is arbitrary and $Q$ is connected, it follows that $R_Z = Q_Z$ or $R_X = Q_Z^{\text{op}}$, see \cite{ASS12}.
Then, $R$ is of type $\At_{p,q}$, which proves that cluster algebras of type $\At$ are unistructural.
\end{proof}

\subsection{Consequence} Recall the definition of a cluster automorphism from \cite{ASS12}. Let $\Aa{X}{Q}$ be a cluster algebra, and let $f : \Aa{X}{Q} \rightarrow \Aa{X}{Q}$ be an automorphism of $\Z$-algebras.
Then $f$ is called a \emph{cluster automorphism} if there exists a seed $(Y,R)$ of $\Aa{X}{Q}$, such that
the following conditions are satisfied: \begin{itemize}
\item $f(Y)$ is a cluster in $\Aa{X}{Q}$;
\item $f$ is compatible with mutations, that is, for every $y_k \in Y$, we have \[f\left(\mu_{y_k}(Y)\right) = \mu_{f(y_k)}(f(Y)).\]
\end{itemize}
 In \cite{Sal14}, the author defines another notion of automorphism of a cluster algebra: this is an automorphism of the ambient field which restricts to a permutation of the set of cluster variables.

\begin{crl}
Let $\Aa{X}{Q}$ be a cluster algebra of type $\At$. Then $f:\Aa{X}{Q} \rightarrow \Aa{X}{Q}$ is a cluster automorphism
if and only if $f$ is an automorphism of the ambient field which restricts to a permutation of the set of cluster variables.
\end{crl}

\begin{proof}
Lemma 5.2 in \cite{ASS14} states that this result is true for unistructural cluster algebras. It then follows from Theorem \ref{yeah}.
\end{proof}

It is conjectured in \cite{ASS14} that this result is true for any cluster algebra. Indeed, it would follow from Conjecture \ref{unisructurality} and Lemma 5.2 in \cite{ASS14}.


\bibliographystyle{alpha}
\bibliography{bibliographie}

\end{document}